\documentclass{article}

\usepackage{amssymb}
\usepackage{amsthm}
\usepackage{bbm}
\usepackage{commath}
\usepackage[hidelinks]{hyperref}
\usepackage{mathtools}

\newtheorem{n}{}[]

\theoremstyle{definition}
\newtheorem{lemma}[n]{Lemma}
\newtheorem{proposition}[n]{Proposition}
\newtheorem{theorem}[n]{Theorem}

\theoremstyle{definition}
\newtheorem{notation}[n]{Notation}
\newtheorem{remark}[n]{Remark}
\newtheorem{example}[n]{Example}

\newcommand{\EEE}{\mathcal{E}}
\newcommand{\et}{\text{\'et}}
\newcommand{\FF}{\mathbb{F}}
\newcommand{\ffff}{\mathfrak{f}}
\newcommand{\FFF}{\mathcal{F}}
\newcommand{\FFFF}{\mathfrak{F}}
\newcommand{\Fr}{\mathrm{Fr}}
\newcommand{\Gal}{\mathrm{Gal}}
\newcommand{\GL}{\mathrm{GL}}
\newcommand{\KKKK}{\mathfrak{K}}
\newcommand{\LLL}{\mathcal{L}}
\newcommand{\LLLL}{\mathfrak{L}}
\newcommand{\NN}{\mathbb{N}}
\newcommand{\NS}{\mathrm{NS}}
\newcommand{\PP}{\mathbb{P}}
\newcommand{\QQ}{\mathbb{Q}}
\newcommand{\st}{\mid}
\newcommand{\tr}{\mathrm{tr}}

\newcommand{\br}{\!\del[0]}
\let\cb\cbr\renewcommand{\cbr}{\!\cb[0]}
\let\sb\sbr\renewcommand{\sbr}{\!\sb[0]}

\title{Algebraicity of Artin--Hasse--Weil L-series over global function fields}
\author{David Kurniadi Angdinata}

\begin{document}

\maketitle

\begin{abstract}
We prove an analogue of Deligne's period conjecture for the special value of the L-function of an abelian variety over a global function field twisted by an Artin representation. We illustrate this in action for an example of an elliptic curve twisted by a Dirichlet character.
\end{abstract}

Deligne's period conjecture is an abstract statement on the special value of the L-function associated to a pure motive with a critical Hodge structure \cite[Definition 1.3]{Del79}. Specifically, he conjectures that the L-value is equal to the determinant of a certain period map between its Betti and de Rham realisations, up to non-zero multiples in a number field \cite[Conjecture 2.8]{Del79}. This is known for Artin L-functions over $ \QQ $ \cite[Proposition 6.7]{Del79}, and has ramifications for the Birch--Swinnerton-Dyer conjecture for abelian varieties over $ \QQ $ \cite[Section 4]{Del79}, with numerical evidence for L-functions associated to Jacobians of smooth projective curves over $ \QQ $ \cite[Conjecture 1.1]{ECW24}.

In the context of the L-function $ L\br{A, \tau, s} $ of an abelian variety $ A $ over a number field $ K $ twisted by an Artin representation $ \tau $ over $ K $, which appear in equivariant refinements of the Birch--Swinnerton-Dyer conjecture \cite[Conjecture 3.3]{BC24}, Deligne's period conjecture translates to a statement on the algebraicity and Galois equivariance of $ L\br{A, \tau, 1} $ normalised by periods \cite[Proposition 4.3.8]{Eva21}. This remains largely open in general, but the case of an elliptic curve over $ \QQ $ twisted by Artin representations that factor through a false Tate curve extension, such as the trivial representation and primitive Dirichlet characters, is a consequence of the modularity theorem \cite[Theorem 4.2]{BD07}.

When $ A $ is an abelian variety over a global function field $ K $, the works of Grothendieck \cite[Theorem 5.1]{Gro95} and Deligne \cite[Theorem 9.3]{Del73} show that $ L\br{A, \tau, s} $ is already a rational function satisfying a globally compatible functional equation, so that the aforementioned normalisations by periods are unnecessary. This paper presents a short proof of the analogue of Deligne's period conjecture in this context, which is stated in Theorem \ref{thm:hasseweil}. To this end, some notational conventions for the formalism of $ \ell $-adic representations over local fields and global function fields will first be established. Throughout, $ \ell $ will be a fixed prime of $ \QQ $, and $ V_\ell $ will be a finite-dimensional vector space over a finite extension of $ \QQ_\ell $, whose choice will not be essential.

\pagebreak

\begin{notation}
Let $ F $ be a non-archimedean local field with residue characteristic $ p $. Let $ I_F $ denote the inertia subgroup of its Weil group $ W_F $, and let $ \Fr_F $ denote the inverse of any choice of Frobenius element in $ W_F $. For $ \ell \ne p $, an \emph{$ \ell $-adic representation over $ F $} is a continuous homomorphism $ \rho : W_F \to \GL\br{V_\ell} $. Its \emph{Euler factor} is the inverse characteristic polynomial
$$ \LLL_F\br{\rho, T} \coloneqq \det\br{1 - T \cdot \Fr_F \st \rho^{I_F}}, $$
where $ \rho^{I_F} $ is the subrepresentation of $ \rho $ invariant under $ I_F $.

Now let $ K $ be the global function field of a smooth proper geometrically irreducible curve $ C $ of genus $ g_C $ over a finite field $ \FF_q $ with absolute Galois group $ G_K $. For each place $ v $ of $ K $, let $ K_v $ denote its completion, and let $ \deg v $ denote its residue class degree. For $ \ell \nmid q $, an \emph{$ \ell $-adic representation over $ K $} is a continuous homomorphism $ \rho : G_K \to \GL\br{V_\ell} $. Its \emph{formal L-series} is the infinite product
$$ \LLL\br{\rho, T} \coloneqq \prod_v \dfrac{1}{\LLL_{K_v}\br{\rho, T^{\deg v}}}, $$
which is a priori only a formal product. Its \emph{L-series} $ L\br{\rho, s} $ is simply $ \LLL\br{\rho, q^{-s}} $, and let $ L^{\br{n}}\br{\rho, s} $ denote the $ n $-th derivative of $ L\br{\rho, s} $ for all $ n \in \NN $. Let $ \ffff_\rho $ denote the Artin conductor of $ \rho $, and let $ \deg\ffff_\rho $ denote its degree as a Weil divisor on $ C $. Finally, let $ G_K^g $ denote the geometric Galois group, namely the kernel of the natural restriction from $ G_K $ to the absolute Galois group of $ \FF_q $, and let $ \rho^{G_K^g} $ denote the subrepresentation of $ \rho $ invariant under $ G_K^g $.

The key example of an $ \ell $-adic representation over $ K $ will be the first $ \ell $-adic cohomology group $ \rho_A \coloneqq H_\et^1\br{A, \QQ_\ell} $ of an abelian variety $ A $ over $ K $, which is independent of $ \ell $ \cite[Theorem 4.3]{GR72}, so $ \ell $ is suppressed from notation. Another example is an \emph{Artin representation}, namely a continuous homomorphism $ \tau : G_K \to \GL\br{V} $, where $ V $ is a finite-dimensional vector space over a number field equipped with the discrete topology, viewed as an $ \ell $-adic representation over $ K $ by some embedding $ \overline{\QQ} \hookrightarrow \overline{\QQ_\ell} $. The relevant notions over $ F $ are defined analogously. Finally, let $ L^{\br{n}}\br{A, \tau, s} $ denote $ L^{\br{n}}\br{\rho_A \otimes \tau, s} $ for all $ n \in \NN $.

For an Artin representation $ \tau $, let $ \QQ\br{\tau} $ denote the number field generated by the values of $ \tr\br{\tau} $, and let $ \tau^\sigma $ denote the representation with character $ \sigma \circ \tr\br{\tau} $ for any $ \sigma \in G_\QQ $. If $ \br{v_i}_i $ is a basis of $ \tau $ over $ \QQ $, and $ \br{a_{ij}} $ is the matrix of $ g \in G_K $ with respect to this basis, then $ \br{v_i^\sigma}_i $ is a basis of $ \tau^\sigma $ over $ \QQ $, and the matrix of $ g $ with respect to this basis is $ \br{a_{ij}^{\sigma}} $ \cite[Section 2.1.4]{Eva21}.
\end{notation}

The main result of this paper is as follows.

\begin{theorem}
\label{thm:hasseweil}
Let $ A $ be an abelian variety over a global function field $ K $, and let $ \tau $ be an Artin representation over $ K $. Then $ L^{\br{n}}\br{A, \tau, 1} \in \QQ\br{\tau} $ and $ L^{\br{n}}\br{A, \tau, 1}^\sigma = L^{\br{n}}\br{A, \tau^\sigma, 1} $ for any $ \sigma \in G_\QQ $, for all $ n \in \NN $.
\end{theorem}

It turns out that the same argument applies to Artin L-series.

\begin{theorem}
\label{thm:artin}
Let $ \tau $ be an Artin representation over $ K $. Then $ L^{\br{n}}\br{\tau, 1} \in \QQ\br{\tau} $ and $ L^{\br{n}}\br{\tau, 1}^\sigma = L^{\br{n}}\br{\tau^\sigma, 1} $ for any $ \sigma \in G_\QQ $, for all $ n \in \NN $.
\end{theorem}

\pagebreak

In what follows, a stronger result on the algebraicity of formal L-series will be proven, which clearly implies the same for all derivatives of L-series, by replacing $ T $ with $ q^{-s} $. To this end, for any field $ \FFFF $ with automorphism group $ G $, define an action of $ G $ on the ring of formal power series $ \FFFF[[T]] $ by
$$ \del[4]{\sum_{n = 0}^\infty a_nT^n}^g \coloneqq \sum_{n = 0}^\infty a_n^g T^n, \qquad g \in G. $$
Evaluating such a power series at an element $ f \in \FFFF $ does not give $ \sum_{n = 0}^\infty a_n^gf^n $ in general, due to potential convergence issues, but the following shows that it does whenever the power series happens to be a rational function.

\begin{lemma}
\label{lem:powerseries}
Let $ \FFFF $ be a field, and let $ P\br{T} \in \FFFF[[T]] $ be a power series such that $ P\br{T} = R\br{T} / Q\br{T} $ for some power series $ Q\br{T} \in \FFFF[[T]] $ and some polynomial $ R\br{T} \in \FFFF\sbr{T} $. Then $ P\br{T}^\sigma = R\br{T}^\sigma / Q\br{T}^\sigma $.
\end{lemma}

\begin{proof}
Since $ R\br{T} $ is a polynomial, it suffices to show that $ \br{P\br{T}Q\br{T}}^\sigma = P\br{T}^\sigma Q\br{T}^\sigma $. Let $ P_n $ and $ Q_n $ denote the coefficients of the power series $ P\br{T} $ and $ Q\br{T} $ respectively for all $ n \in \NN $, so that the equality becomes
$$ \sum_{n = 0}^\infty \del[4]{\sum_{i + j = n} P_iQ_j}^\sigma T^n = \sum_{n = 0}^\infty P_n^\sigma T^n \cdot \sum_{n = 0}^\infty Q_n^\sigma T^n. $$
This is clear since $ \sum_{i + j = n} P_iQ_j $ is a finite sum.
\end{proof}

This is a property inherited by the formal L-series of a general $ \ell $-adic representation over a global function field. Furthermore, if it has no geometric invariants, its formal L-series is in fact a polynomial.

\begin{proposition}
\label{prop:rationality}
Let $ \rho $ be an $ \ell $-adic representation over a global function field $ K = \FF_q\br{C} $ that is unramified almost everywhere. Then $ \LLL\br{\rho, T} \in \overline{\QQ_\ell}\br{T} $. Furthermore, if $ \rho^{G_K^g} = 0 $, then $ \LLL\br{\rho, T} \in \overline{\QQ_\ell}\sbr{T} $ of degree
$$ \deg\LLL\br{\rho, T} = \br{2g_C - 2}\dim\rho + \deg\ffff_\rho. $$
\end{proposition}

\begin{proof}
This follows the sketch of an argument in Ulmer's notes \cite[Lecture 4, Theorem 1.4.1]{Ulm11}, but it is repeated here with references to Milne's book. There is an equivalence of categories between continuous $ \ell $-adic representations over $ K $ that are unramified on an open set $ U $ of $ C $ and $ \ell $-adic sheaves that are lisse on $ U $ \cite[Chapter V, Section 1]{Mil80}. Let $ \iota : U \hookrightarrow C $ be any open set at which $ \rho $ is unramified, and let $ \FFF_\rho $ be its associated $ \ell $-adic sheaf that is lisse on $ U $, whose direct image along $ \iota $ induces \'etale cohomology groups $ H^i \coloneqq H_\et^i\br{\overline{C}, \iota_*\FFF_\rho} $ of the base change $ \overline{C} $ of $ C $ to $ \overline{\FF_q} $. Now the Grothendieck--Lefschetz trace formula for $ \ell $-adic sheaves \cite[Chapter VI, Theorem 13.4]{Mil80} says that for all $ n \in \NN $,
$$ \sum_{v \in C\br{\FF_{q^n}}} \tr\br{\Fr_{K_v}^n \st \rho^{I_{K_v}}} = \sum_{i = 0}^2 \br{-1}^i \cdot \tr\br{\Fr_q^n \st H^i}, $$

\pagebreak

\noindent where $ \Fr_q $ is the Frobenius in $ \FF_q $. Dividing both sides by $ n $ and exponentiating their generating functions, this equality rearranges to
$$ \prod_v \exp\sum_{m = 1}^\infty \tr\br{\Fr_{K_v}^m \st \rho^{I_{K_v}}}\dfrac{T^{m\deg v}}{m} = \prod_{i = 0}^2 \exp\del[4]{\sum_{n = 1}^\infty \tr\br{\Fr_q^n \st H^i}\dfrac{T^n}{n}}^{\br{-1}^i}. $$
An identity in linear algebra \cite[Chapter V, Lemma 2.7]{Mil80} shows that
$$ \exp\sum_{m = 1}^\infty \tr\br{\Fr_{K_v}^m \st \rho^{I_{K_v}}}\dfrac{T^{m\deg v}}{m} = \dfrac{1}{\det\br{1 - T^{\deg v} \cdot \Fr_{K_v} \st \rho^{I_{K_v}}}}, $$
for each place $ v $ of $ K $, and that
$$ \exp\sum_{n = 1}^\infty \tr\br{\Fr_q^n \st H^i}\dfrac{T^n}{n} = \dfrac{1}{\det\br{1 - T \cdot \Fr_q \st H^i}}. $$
for $ i = 0, 1, 2 $. Thus the left hand side becomes $ \LLL\br{\rho, T} $, while the right hand side expresses it as an alternating product of polynomials $ \det\br{1 - T \cdot \Fr_q \st H^i} $, which proves the first statement. For the second statement, note that
$$ H^0 = H_\et^0\br{\overline{U}, \FFF_\rho} \cong \rho^{G_K^g}, $$
by definition, and Poincar\'e duality \cite[Chapter V, Proposition 2.2(c)]{Mil80} gives
$$ H^2 \cong H_\et^0\br{\overline{C}, \iota_*\FFF_\rho^\vee\br{1}}^\vee = H_\et^0\br{\overline{U}, \FFF_\rho^\vee\br{1}}^\vee \cong \br{\rho^\vee\br{1}^{G_K^g}}^\vee = 0, $$
so that $ \LLL\br{\rho, T} \in \overline{\QQ_\ell}\sbr{T} $. Its degree is precisely given by the Grothendieck--Ogg--Shafarevich formula \cite[Chapter V, Theorem 2.12]{Mil80}.
\end{proof}

In particular, $ \LLL\br{\rho, T} $ is well-defined at $ T = q^{-1} $ and respects the action of automorphisms of $ \overline{\QQ_\ell} $ whenever its denominator does not vanish.

\begin{remark}
When $ \rho $ is the $ \ell $-adic representation associated to an elliptic curve, Shioda gave an alternative description of $ \LLL\br{\rho, T} $ in terms of its associated elliptic surface $ \EEE $ \cite[Theorem 4]{Shi92}. When $ \EEE $ has at least one singular fibre, he showed that $ \LLL\br{\rho, T} $ is in fact a polynomial, given by
$$ \LLL\br{\rho, T} = \det\br{1 - T \cdot \Fr_q \st W}, $$
where $ W $ is a subspace of the second $ \ell $-adic cohomology group $ H_\et^2\br{\EEE, \QQ_\ell\br{1}} $ of $ \EEE $, given as the orthogonal complement of the trivial sublattice of the Neron--Severi group $ \NS\br{\EEE} $ of $ \EEE $ under the cycle class map $ \NS\br{\EEE} \to H_\et^2\br{\EEE, \QQ_\ell\br{1}} $. His description has the added benefit that the degree and functional equation of the polynomial $ \LLL\br{\rho, T} $ can be understood directly from the geometry of $ \EEE $.
\end{remark}

The analogue of algebraicity and Galois equivariance can first be proven at the level of Euler factors of local $ \ell $-adic representations.

\pagebreak

\begin{proposition}
\label{prop:local}
Let $ \rho $ be an $ \ell $-adic representation over a non-archimedean local field $ F $, such that $ \LLL_F\br{\rho, T} $ has coefficients in $ \QQ $, and let $ \tau $ be an Artin representation over $ F $. Then $ \LLL_F\br{\rho \otimes \tau, T} \in \QQ\br{\tau}\sbr{T} $ and $ \LLL_F\br{\rho \otimes \tau, T}^\sigma = \LLL_F\br{\rho \otimes \tau^\sigma, T} $ for any $ \sigma \in G_\QQ $.
\end{proposition}

\begin{proof}
This is similar to the argument by Bouganis--Dokchitser, but they only proved it for Artin twists of elliptic curves over number fields \cite[Lemma 4.4]{BD07}, so it is repeated here for reference. The first statement follows from the second statement since $ \tau^\sigma = \tau $ for any $ \sigma \in G_{\QQ\br{\tau}} $, so it suffices to prove the latter. Since $ \LLL_F\br{\rho \otimes \tau, T} $ has coefficients in $ \QQ\br{\tau} $, it suffices to prove it for $ \sigma \in \Gal\br{\LLLL / \KKKK} $, where $ \LLLL $ is the extension of $ \QQ\br{\tau} $ that realises $ \tau $ and $ \KKKK $ is its subfield fixed by $ \sigma $. There is an equivalence of categories between $ \ell $-adic representations over $ F $ and complex Weil--Deligne representations of $ W_F $ \cite[Section 8]{Del73}, so that $ \ell $ can be replaced with some prime $ \ell' $ that splits in $ \KKKK $ and remains inert in $ \LLLL $, which exists by Chebotarev's density theorem. This gives an isomorphism $ \phi : \Gal\br{\LLLL / \KKKK} \xrightarrow{\sim} \Gal\br{\QQ_{\ell'}\br{\alpha} / \QQ_{\ell'}} $, where $ \alpha $ is the image in $ \overline{\QQ_{\ell'}} $ of the primitive element of $ \KKKK $ that generates $ \LLLL $. Now let $ \br{v_i}_i $ be a basis of $ \br{\rho \otimes \tau}^{I_F} $ over $ \QQ_\ell $, and let $ \br{a_{ij}} $ be the matrix of $ \Fr_F $ with respect to this basis. Then $ \br{v_i^{\phi\br{\sigma}}}_i $ is a basis of $ \br{\rho \otimes \tau^\sigma}^{I_F} $ over $ \QQ_\ell $, and the matrix of $ \Fr_F $ with respect to this basis is $ \br{a_{ij}^{\phi\br{\sigma}}} $, so that its inverse characteristic polynomial is precisely that of $ \br{a_{ij}}^\sigma $.
\end{proof}

The corresponding statement for formal L-series follows from the local statements, by rewriting the infinite product of local Euler factors into a power series with coefficients indexed by effective Weil divisors, and applying rationality.

\begin{theorem}
\label{thm:global}
Let $ \rho $ be an $ \ell $-adic representation over a global function field $ K = \FF_q\br{C} $ that is unramified almost everywhere, such that $ \LLL_{K_v}\br{\rho, T} $ has coefficients in $ \QQ $ for each place $ v $ of $ K $, and let $ \tau $ be an Artin representation over $ K $. Then $ \LLL\br{\rho \otimes \tau, T} \in \QQ\br{\tau}\br{T} $ and $ \LLL\br{\rho \otimes \tau, T}^\sigma = \LLL\br{\rho \otimes \tau^\sigma, T} $ for any $ \sigma \in G_\QQ $. Furthermore, if $ \br{\rho \otimes \tau}^{G_K^g} = 0 $, then $ \LLL\br{\rho \otimes \tau, T} \in \QQ\br{\tau}\sbr{T} $ of degree
$$ \deg\LLL\br{\rho \otimes \tau, T} = \br{2g_C - 2}\dim\rho\dim\tau + \deg\ffff_{\rho \otimes \tau}. $$
\end{theorem}

\begin{proof}
For each place $ v $ of $ K $, let $ a_{v, n}\br{\tau} $ denote the coefficients of the power series $ \LLL_{K_v}\br{\rho \otimes \tau, T}^{-1} $ for all $ n \in \NN $. By Lemma \ref{lem:powerseries} for $ P\br{T} = \LLL_{K_v}\br{\rho \otimes \tau, T}^{-1} $, Proposition \ref{prop:local} translates into $ a_{v, n}\br{\tau} \in \QQ\br{\tau} $ and $ a_{v, n}\br{\tau}^\sigma = a_{v, n}\br{\tau^\sigma} $ for any $ \sigma \in G_\QQ $. Now for an effective Weil divisor $ D = \sum_v n_v\sbr{v} $ on $ C $, let $ a_D\br{\tau} $ denote the finitely-supported product $ \prod_v a_{v, n_v}\br{\tau} $, so that $ a_D\br{\tau} \in \QQ\br{\tau} $ and $ a_D\br{\tau}^\sigma = a_D\br{\tau^\sigma} $ for any $ \sigma \in G_\QQ $. A rearrangement gives
$$ \LLL\br{\rho \otimes \tau, T} = \prod_v \sum_{n = 0}^\infty a_{v, n}\br{\tau}T^{n\deg v} = \sum_{m = 0}^\infty \sum_D a_D\br{\tau}T^m, $$
where the sum ranges over effective Weil divisors $ D $ on $ C $ of degree precisely $ m $. This is a finite sum, so that $ \sum_D a_D\br{\tau} \in \QQ\br{\tau} $ and $ \br{\sum_D a_D\br{\tau}}^\sigma = \sum_D a_D\br{\tau^\sigma} $, which proves the second statement and that $ \LLL\br{\rho \otimes \tau, T} \in \QQ\br{\tau}[[T]] $. The first and final statements follow from Proposition \ref{prop:rationality} that $ \LLL\br{\rho \otimes \tau, T} \in \overline{\QQ_\ell}\br{T} $, using the theory of Hankel determinants \cite[Chapter IV.4, Exercise 1]{Bou03}.
\end{proof}

\pagebreak

In particular, these apply to $ \rho = \rho_A $, which proves Theorem \ref{thm:hasseweil}, but also to the trivial representation $ \rho = \mathbbm{1} $, which proves Theorem \ref{thm:artin}.

\begin{remark}
Using Proposition \ref{prop:rationality}, Burns--Kakde--Kim proves the algebraicity and Galois equivariance of $ L^{\br{n}}\br{A, \tau, s} $ up to finitely many local Euler factors away from an open set $ U $ of $ C $ \cite[Proposition 2.2]{BKK18}, by directly arguing that the action of $ \Fr_q $ is preserved under an isomorphism
$$ H_{\et, c}^i\br{\overline{U}, \FFF_{\rho_A} \otimes \FFF_\tau^\sigma} \cong H_{\et, c}^i\br{\overline{U}, \FFF_{\rho_A} \otimes \FFF_\tau}^\sigma, $$
for any $ \sigma \in G_\QQ $, where both sides are compactly-supported \'etale cohomology groups of the base change $ \overline{U} $ of $ U $ to $ \overline{\FF_q} $. The remaining finitely many local Euler factors can be handled separately by Proposition \ref{prop:local}, which gives an alternative proof for Theorem \ref{thm:hasseweil} independent from Theorem \ref{thm:global}.
\end{remark}

\begin{remark}
There are explicit bounds for $ \deg\ffff_{\rho \otimes \tau} $ in terms of $ \deg\ffff_\rho $ and $ \deg\ffff_\tau $, such as in the arguments of Bisatt--Paterson \cite[Section 2]{BP23}, so the computation of $ \deg\LLL\br{\rho \otimes \tau, T} $ generalises that by Comeau-Lapointe--David--Lalin--Li for Dirichlet twists of elliptic curves \cite[Theorem 2.2]{CLDLL22}.
\end{remark}

The expression for $ \deg\LLL\br{\rho \otimes \tau, T} $ in Theorem \ref{thm:global} is useful for computing formal L-series of Dirichlet twists of elliptic curves, which was done explicitly by Comeau-Lapointe--David--Lalin--Li using the functional equation \cite[Section 5.1]{CLDLL22}. The reader is referred to Ulmer's notes \cite[Lecture 1]{Ulm11} and Rosen's book \cite[Chapter 4 and Chapter 9]{Ros02} for the general theory over global function fields of elliptic curves and Dirichlet characters respectively.

\begin{example}
Let $ K = \FF_{11}\br{t} $, and let $ A $ be the elliptic curve over $ K $ given by $ Y^2 = X^3 + \br{t + 1}^3\br{t + 2}^3 $. Consider the Dirichlet character $ \tau $ over $ K $ of modulus $ t $ given by $ 2 \mapsto \zeta_5 $ and the automorphism $ \sigma \in \Gal\br{\QQ\br{\zeta_5} / \QQ} $ given by $ \zeta_5 \mapsto \zeta_5^2 $. To verify Theorem \ref{thm:hasseweil} that $ \LLL\br{A, \tau, T}^\sigma = \LLL\br{A, \tau^\sigma, T} $, first compute
$$ \LLL\br{A, \tau, T} = 1 + 11\br{5 + 12\zeta_5 + 5\zeta_5^2}T^2 + 14641\zeta_5^2T^4. $$
Since $ \tau^\sigma = \sigma \circ \tau $ is given by $ 2 \mapsto \zeta_5^2 $, separately compute
$$ \LLL\br{A, \tau^\sigma, T} = 1 + 11\br{5 + 12\zeta_5^2 + 5\zeta_5^4}T^2 + 14641\zeta_5^4T^4. $$
These were computed in Magma, but the same algorithm works in any software with support for irreducible polynomials over finite fields. Note that $ \ffff_{\rho_A} = 2\sbr{t + 1} + 2\sbr{t + 2} $ and $ \ffff_\tau = \sbr{t} + \sbr{1 / t} $ have disjoint support, so that $ \br{\rho_A \otimes \tau}^{I_{K_t}} = 0 $. In particular, $ \br{\rho_A \otimes \tau}^{G_K^g} = 0 $, so that $ \LLL\br{A, \tau, T} \in \QQ\br{\zeta_5}\sbr{T} $ of degree
$$ \deg\LLL\br{A, \tau, T} = \br{2g_{\PP^1} - 2}\dim\rho_A\dim\tau + \deg\ffff_\tau\dim\rho_A + \deg\ffff_{\rho_A}\dim\tau = 4. $$
Thus $ \LLL\br{A, \tau, T} $ is completely determined by $ \LLL_{K_v}\br{A, \tau, T} $ for all places $ v $ of $ K $ with $ \deg v \le 4 $, where $ \LLL_{K_v}\br{A, \tau, T} = 1 $ for $ v \in \cbr{1 / t, t, t + 1, t + 2} $ and
$$ \LLL_{K_v}\br{A, \tau, T} = 1 - \tr\br{\Fr_{K_v} \st \rho_A}\tau\br{v}T + 11^{\deg v}\tau\br{v}^2T^2, $$
for all other places $ v $ of $ K $. For instance, if $ v = t^4 + t + 2 $, then $ \tr\br{\Fr_{K_v} \st \rho_A} = -242 $ and $ \tau\br{v} = \zeta_5 $, so that $ \LLL_{K_v}\br{A, \tau, T} = 1 + 242\zeta_5T + 14641\zeta_5^2T^2 $.
\end{example}

\pagebreak

\begin{remark}
The existence of a globally compatible functional equation can drastically cut down the number of computations of local Euler factors necessary to completely determine the formal L-series. This in turn involves computing Langlands--Deligne local constants, which will not be explored in this paper.
\end{remark}

\def\abstractname{Acknowledgements}

\begin{abstract}
I would like to thank Vladimir Dokchitser for guidance throughout. This work was supported by the Engineering and Physical Sciences Research Council [EP/S021590/1], the EPSRC Centre for Doctoral Training in Geometry and Number Theory (The London School of Geometry and Number Theory), University College London.
\end{abstract}

\bibliography{main}

\end{document}